\documentclass[10pt,a4paper]{article}
\usepackage{amsthm}
\usepackage{amsmath}
\usepackage{amssymb}
\usepackage{amsfonts}
\usepackage{relsize}
\usepackage{geometry}
\usepackage{url}
\usepackage{enumerate}

\begin{document}

\setlength{\parskip}{3mm}

\newtheorem{thm}{Theorem}
\newtheorem{lem}[thm]{Lemma}
\newtheorem{prop}[thm]{Proposition}
\newtheorem{cor}[thm]{Corollary}
\newtheorem{hyp}[thm]{Hypothesis}
\newtheorem{que}{Question}

\theoremstyle{definition}
\newtheorem{defn}[thm]{Definition}

\newcommand{\conda}{(*)}
\newcommand{\condb}{(**)}

\newcommand{\bC}{\mathbb{C}}
\newcommand{\bF}{\mathbb{F}}
\newcommand{\bN}{\mathbb{N}}
\newcommand{\bP}{\mathbb{P}}
\newcommand{\bQ}{\mathbb{Q}}
\newcommand{\bR}{\mathbb{R}}
\newcommand{\bZ}{\mathbb{Z}}

\newcommand{\mcA}{\mathcal{A}}
\newcommand{\mcB}{\mathcal{B}}
\newcommand{\mcC}{\mathcal{C}}
\newcommand{\mcD}{\mathcal{D}}
\newcommand{\mcE}{\mathcal{E}}
\newcommand{\mcF}{\mathcal{F}}
\newcommand{\mcG}{\mathcal{G}}
\newcommand{\mcH}{\mathcal{H}}
\newcommand{\mcI}{\mathcal{I}}
\newcommand{\mcJ}{\mathcal{J}}
\newcommand{\mcK}{\mathcal{K}}
\newcommand{\mcL}{\mathcal{L}}
\newcommand{\mcM}{\mathcal{M}}
\newcommand{\mcN}{\mathcal{N}}
\newcommand{\mcO}{\mathcal{O}}
\newcommand{\mcP}{\mathcal{P}}
\newcommand{\mcQ}{\mathcal{Q}}
\newcommand{\mcR}{\mathcal{R}}
\newcommand{\mcS}{\mathcal{S}}
\newcommand{\mcT}{\mathcal{T}}
\newcommand{\mcU}{\mathcal{U}}
\newcommand{\mcV}{\mathcal{V}}
\newcommand{\mcW}{\mathcal{W}}
\newcommand{\mcX}{\mathcal{X}}
\newcommand{\mcY}{\mathcal{Y}}
\newcommand{\mcZ}{\mathcal{Z}}

\title{On (hereditarily) just infinite profinite groups that are not virtually pro-$p$}

\author{Colin D. Reid\\
Mathematisches Institut\\
Georg-August Universit\"{a}t G\"{o}ttingen\\
Bunsenstra\ss{}e 3-5, D-37073 G\"{o}ttingen\\
Germany\\
colin@reidit.net}

\maketitle

\begin{abstract}A profinite group $G$ is just infinite if every non-trivial closed normal subgroup of $G$ is of finite index, and hereditarily just infinite if every open subgroup is just infinite.  Hereditarily just infinite profinite groups need not be virtually pro-$p$, as shown in a recent paper of Wilson.  The same paper gives a criterion on an inverse system of finite groups that is sufficient to ensure the limit is either virtually abelian or hereditarily just infinite.  We give criteria of a similar nature that characterise the just infinite and hereditarily just infinite properties under the assumption that $G$ is not virtually pro-$p$.\end{abstract}

\begin{defn}In this paper, all groups will be profinite groups, all homomorphisms are required to be continuous, and all subgroups are required to be closed; the notation $[A,B]$ is understood to mean the closure of the commutator of $A$ and $B$.  We use `pronilpotent' to mean a group that is the inverse limit of finite nilpotent groups.  A profinite group $G$ is \emph{just infinite} if it is infinite, and every non-trivial normal subgroup of $G$ is of finite index; it is \emph{hereditarily just infinite} if in addition $H$ is just infinite for every open subgroup $H$ of $G$.\end{defn}

Most results to date on just infinite profinite groups, and especially on hereditarily just infinite profinite groups, concern those which are pro-$p$ or at least virtually pro-$p$.  It is easy to see that such groups are the only just infinite virtually pronilpotent groups.  By contrast a recent paper of J. S. Wilson (\cite{Wil}) gives the first known constructions of hereditarily just infinite profinite groups that are not virtually pro-$p$.  The present paper expands on a result in \cite{Wil} that gives a sufficient condition on an inverse system of finite groups for the limit to be hereditarily just infinite.

\begin{thm}[Wilson \cite{Wil} 2.2]\label{wilthm} Let $G$ be the inverse limit of a sequence $(G_n)_{n \geq 0}$ of finite groups and surjective homomorphisms $G_n \rightarrow G_{n-1}$.  For each $n \geq 1$ write $K_n = \ker(G_n \rightarrow G_{n-1})$, and suppose that for all $L \unlhd G_n$ such that $L \not\le K_n$ the following assertions hold:
\begin{enumerate}[(i)]  \itemsep0pt
\item $K_n \leq L$;
\item $L$ has no proper subgroup whose distinct $G_n$-conjugates centralise each other and generate $L$.
\end{enumerate}
Then $G$ is a just infinite profinite group and is either virtually abelian or hereditarily just infinite.\end{thm}

We give similar conditions on an inverse system which ensure that the limit is just infinite or respectively hereditarily just infinite (Theorem \ref{thma}).  This is not quite a direct generalisation of Theorem \ref{wilthm}, not least because Theorem \ref{wilthm} allows for some pronilpotent groups such as $\bZ_p$ whereas the inverse limits in Theorem \ref{thma} are never virtually pronilpotent, but the results are closely related (see Proposition \ref{thmatowil}).  More significantly, the conditions in Theorem \ref{thma} turn out to characterise just infinite or hereditarily just infinite profinite groups that are not virtually pro-$p$, in that every such group is the limit of an inverse system of the prescribed form; indeed, one can impose apparently stronger conditions (Theorem \ref{thmb}).

\begin{defn}Let $G$ be a profinite group and let $A$ and $B$ be normal subgroups of $G$ such that $B < A$.  Say $A/B$ is a \emph{chief factor} of $G$ if there are no normal subgroups of $G$ lying strictly between $A$ and $B$.  Say $(A,B)$ is a \emph{critical pair} in $G$ if $B$ contains every normal subgroup of $G$ that is properly contained in $A$.\end{defn}

Note that given a critical pair $(A,B)$, then $A/B$ is always a chief factor of $G$, and that if $(A,B)$ is critical in $G$, then $(A/N,B/N)$ is critical in $G/N$ for any $N \unlhd G$ such that $N \leq B$.  Critical pairs have a further useful property concerning centralisers that can be used to establish the (hereditary) just infinite property.

\begin{lem}\label{critlem}Let $G$ be a finite group with critical pair $(A,B)$, and let $K$ be a normal subgroup of $G$ that is not contained in $C_G(A/B)$.  Then $K \geq A$ and $K$ is not nilpotent.\end{lem}
\begin{proof}As $K$ does not centralise $A/B$, it follows that $[A,K] \not\le B$.  But $[A,K]$ is a normal subgroup of $G$ contained in $A$; hence $[A,K] = A$ since $(A,B)$ is critical in $G$.  In turn, the equation $[A,K]=A$ means that all terms of the lower central series of $K$ contain $A > 1$, so $K \geq A$ and $K$ is not nilpotent.\end{proof}

\begin{thm}\label{thma}Let $G$ be the inverse limit of an inverse system $\Lambda = (G_n)_{n\geq0}$ of finite groups, where $G_n$ contains a specified normal subgroup $A_n$, with associated surjective homomorphisms $\rho_n: G_{n+1} \rightarrow G_n$.  Write $P_n = \rho_n(A_{n+1})$, and write $B_{n+1} = \ker\rho_n$.  Suppose that $(A_n,B_n)$ is a critical pair in $G_n$, and that $P_nC_{G_n}(P_n) \leq B_n$, for all but finitely many $n$.  Then $G$ is just infinite and not virtually pronilpotent.

Suppose that in addition, the following condition holds:

$\conda$ For infinitely many $n$, if $U$ is a subgroup of $G_n$ whose distinct $G_n$-conjugates centralise each other and generate a subgroup of $G$ containing $A_n$, then $U \unlhd G_n$.

Then $G$ is hereditarily just infinite.\end{thm}

\begin{proof}Let $\pi_n$ be the homomorphism from $G$ to $G_n$ associated with the inverse limit construction.  Let $T$ be a non-trivial normal subgroup of $G$, and let $T_n = \pi_n(T)$.  By construction, the subgroups $\pi^{-1}_n(A_n)$ of $G$ form a descending chain of subgroups with trivial intersection.  Thus for $n$ sufficiently large, $T_n$ is not contained in $A_n$.  It follows that for $n$ sufficiently large, $T_n$ does not centralise $P_n$, in other words $T_{n+1}$ does not centralise the section $A_{n+1}/B_{n+1}$ of $G_{n+1}$, and $(A_{n+1},B_{n+1})$ is critical in $G_{n+1}$.  Lemma \ref{critlem} now implies that $T_{n+1} \geq A_{n+1}$ and $T_{n+1}$ is not nilpotent.  As $T_{n+1} \geq A_{n+1}$ for all $n$ sufficiently large, it follows that $T \geq \pi^{-1}_n(A_n)$ for some $n$, so $T$ is an open non-pronilpotent subgroup of $G$.  Thus $G$ is just infinite and not virtually pronilpotent.

Suppose $G$ is not hereditarily just infinite; then $G$ has a non-trivial subgroup $U$ of infinite index, such that the distinct conjugates of $U$ centralise each other (see e.g. \cite{Wil} 2.1).  Let $L$ be the normal closure of $U$ in $G$.  Since $L$ is a non-trivial normal subgroup of $G$ and hence open, there is some $n$ such that $\pi_m(L)$ contains $A_m$ for all $m \geq n$.  Moreover, the distinct $G_m$-conjugates of $\pi_m(U)$ centralise each other, and for $m$ sufficiently large, $\pi_m(U)$ is not normal in $G_m$ since $U$ is not normal in $G$, contradicting $\conda$.\end{proof}

There is a direct connection here with Theorem \ref{wilthm} that is not proved via the just infinite property, as follows:

\begin{prop}\label{thmatowil}Let $\Lambda$ be an inverse system $\Lambda = (G_n)_{n \geq 0}$ of the form described in Theorem \ref{wilthm}, with $K_n$ as specified.  Suppose in addition that $K_n$ has a non-central minimal normal subgroup for all but finitely many $n$.  Then it is possible to specify $A_n \leq G_n$ so that $\Lambda$ of the form specified in Theorem \ref{thma}, including condition $\conda$.\end{prop}

\begin{proof}Let $S_n = \ker(G \rightarrow G_{n-1})$ for $n \geq 1$.  Condition (i) of Theorem \ref{wilthm} can be reinterpreted as follows: Given any normal subgroup $L$ of $G$, either $LS_{n+1}$ properly contains $S_n$, or else $S_n$ contains $L$.

Let $R$ be such that $R/S_n$ is minimal normal in $G/S_n$, and let $L$ be a normal subgroup of $G$ that is properly contained in $R$.  Suppose $L$ is not contained in $S_n$.  Then $LS_n = R$, and there is some $i\geq 0$ such that $LS_{n+i} = R$ but $LS_{n+i+1} < R$.  In particular, $LS_{n+i+1}$ does not contain $S_{n+i}$.  But then $S_{n+i}$ contains $L$, so $S_{n+i}=R$, a contradiction.  Hence $(R,S_n)$ is a critical pair in $G$.

For some $n$, suppose for all $R$ as above that $C_G(R/S_n)$ contains $S_{n-1}$.  Then every minimal normal subgroup of $S_{n-1}/S_n \cong K_{n-1}$ is central.  By assumption, this can only occur for finitely many $n$.  Thus there is some $R_n > S_n$ such that $(R_n,S_n)$ is critical in $G$ and such that, for $n$ sufficiently large, $C_G(R_n/S_n)$ does not contain $S_{n-1}$.  In this case $C_G(R_n/S_n)$ is contained in $S_{n-1}$.

Now set $A_n = R_n/S_{n+1}$, set $P_n = R_{n+1}/S_{n+1}$, and set $B_n = S_n/S_{n+1} = K_n$.  It is clear that $B_n$ and $P_n$ arise in the specified way from the series $A_n$ and the kernels of the maps $G_n \rightarrow G_{n-1}$, that $(A_n,B_n)$ is a critical pair in $G_n$ and that $P_nC_{G_n}(P_n) \leq B_n$ for all but finitely many $n$.  Moreover, by condition (ii) of Theorem \ref{wilthm}, if $L$ is a normal subgroup of $G_n$ properly containing $B_n=K_n$, then $L$ is not the product of a conjugacy class of non-normal subgroups of $G_n$ that centralise each other.  This demonstrates condition $\conda$, since $A_n > B_n$.\end{proof}

It remains to show that any (hereditarily) just infinite profinite group $G$ has an inverse system of the form given in Theorem \ref{thma}.  Moreover, given information about the composition factors of $G$, it is possible to impose some related conditions over the isomorphism types of simple groups that appear in the `critical chief factors' $A_n/B_n$, although this is complicated somewhat by the existence of perfect central extensions of finite simple groups.  (Recall here the well-known theorem of Schur that a perfect finite group $G$ has a universal perfect central extension $Q$, and the centre of $Q$ is finite; in this case $Z(Q)$ is the Schur multiplier of $G$.  To see why this is relevant to the present situation, consider the abelian chief factors of a descending iterated wreath product in which the wreathing groups are proper perfect central extensions of finite simple groups.)

\begin{defn}Let $G$ be a profinite group and let $\mcC$ be a class of finite simple groups (possibly including cyclic groups of prime order).  Define the following condition:

$\condb$ For every integer $n$, there is a finite image $G/N$ of $G$ and a composition series for $G/N$ in which at least $n$ of the factors are in $\mcC$.  For every prime $p$, if $\mcC$ contains the cyclic group of order $p$, then $\mcC$ also contains all non-abelian finite simple groups whose Schur multipliers have order divisible by $p$.\end{defn}

\begin{defn}A finite group $G$ is a \emph{central product} of subgroups $\{H_i \mid i \in I\}$ if these subgroups generate $G$, and whenever $i \not=j$ then $[H_i,H_j]=1$.  Say $G$ is \emph{centrally decomposable} if it is a central product of proper subgroups.\end{defn}

\begin{thm}\label{thmb}Let $G$ be a just infinite profinite group that is not virtually pro-$p$.  Let $\mcC$ be a class of finite simple groups such that $\condb$ holds.  Then $G$ is the limit of an inverse system of the form $\Lambda$ described below.

Let $\Lambda = (G_n)_{n\geq0}$ be an inverse system of finite groups, where $G_n$ contains a specified normal subgroup $A_n$, with associated surjective homomorphisms $\rho_n: G_{n+1} \rightarrow G_n$.  Write $P_n = \rho_n(A_{n+1})$, and write $B_{n+1} = \ker\rho_n$.  The pair $(A_n,B_n)$ is the image of a critical pair in $G$, so in particular is critical in $G_n$, and $A_n/B_n$ is a direct power of a group in $\mcC$.  Furthermore $P_nC_{G_n}(P_n) \leq B_n$.

If $G$ is hereditarily just infinite, one may additionally arrange that every subgroup of $G_n$ normalised by $A_n$ either contains $P_nC_{G_n}(P_n)$, or is contained in every maximal normal subgroup of $A_n$ (or both); in this case, no normal subgroup of $G_n$ containing $A_n$ is centrally decomposable.\end{thm}

The main idea in the proof of Theorem \ref{thmb} is (generalised) obliquity in the sense of \cite{Rei}; this will be used to obtain suitable critical pairs.

\begin{thm}[see \cite{Rei} Theorem A]\label{genob} Let $G$ be a just infinite profinite group, and let $H$ be an open subgroup of $G$.  Then $H$ contains all but finitely many of the normal subgroups of $G$.\end{thm}

\begin{lem}\label{minlem}Let $G$ be a just infinite profinite group, let $K$ and $L$ be normal subgroups such that $L < K$, and suppose that $K/L$ is a chief factor of $G$.  Then there is a critical pair $(A,A \cap L)$ in $G$ such that $AL=K$.  Note in particular that $K/L \cong A/(A \cap L)$ and $C_G(K/L) = C_G(A/(A\cap L))$.\end{lem}
\begin{proof}By Theorem \ref{genob}, the collection $\mcK$ of normal subgroups of $G$ not contained in $L$ is finite.  As $K$ is an element of $\mcK$, there is a minimal element $A$ of $\mcK$ contained in $K$.  It follows that $AL > A$ and $AL \leq K$, so $AL=K$.  Since $A$ is minimal in $\mcK$, any normal subgroup of $G$ properly contained in $A$ must be contained in $L$, and hence in $A \cap L$.  Thus $(A,A \cap L)$ is a critical pair.\end{proof}

\begin{defn}Given a profinite group $G$ and a prime $p$, write $E^p(G)$ for the intersection of all normal subgroups of $G$ of index $p$.\end{defn}

\begin{lem}\label{opsch}Let $G$ be a finite group and let $p$ be a prime.  Suppose that $G$ has a chief factor of exponent $p$, that all chief factors of $G$ of exponent $p$ are central, and that $p$ does not divide the order of the Schur multiplier of any non-abelian composition factor of $G$.  Then $E^p(G) < G$.\end{lem}
\begin{proof}Let $N$ be a normal subgroup of largest order such that $E^p(N) < N$.  Such an $N$ exists by the fact that $G$ has a chief factor of exponent $p$.  Suppose $N < G$, and let $K/N$ be a minimal normal subgroup of $G/N$.  If $K/N$ is abelian, then $[K,K] \leq N$ and $[N,K] \leq E^p(N)$, so $K/E^p(N)$ is nilpotent.  On the other hand, if $K/N$ is non-abelian, then it is a direct power of a non-abelian finite simple group $S$, such that the Schur multiplier of $S$ has order coprime to $p$.  It follows that $K/[K,K]E^p(N)$ is a non-trivial $p$-group.  In either case $E^p(K) < K$, contradicting the choice of $N$.\end{proof}

\begin{lem}\label{seclem}Let $G$ be a just infinite profinite group that is not virtually pronilpotent.  Let $\mcC$ be a class of finite simple groups such that $\condb$ holds.  Let $H$ be an open subgroup of $G$.  Let $\mcD$ be the set of critical pairs $(A,B)$ in $G$ such that $A/B$ is a direct power of a group in $\mcC$.  Then there are infinitely many pairs $(A,B) \in \mcD$ such that $AC_G(A/B)$ is contained in $H$.\end{lem}
\begin{proof}Given a normal subgroup $R$ of $G$, let $S$ be the smallest normal subgroup of $R$ for which $R/S$ has no composition factors in $\mcC$.  Then $S$ is characteristic in $R$ and hence normal in $G$, and $G/S$ has only finitely many composition factors in $\mcC$.  Hence $S > 1$, so $S$ is of finite index in $G$.  Moreover, by construction, if $L$ is a maximal proper $G$-invariant subgroup of $S$, then all the composition factors of $S/L$ are in $\mcC$.  By letting $R$ range over the normal subgroups of $G$ contained in $H$ and applying Lemma \ref{minlem}, one obtains infinitely many pairs $(A,B) \in \mcD$ such that $A$ is contained in $H$.  Indeed, such critical pairs account for every $\mcC$-composition factor of all finite images of $H$.

It now suffices to assume that for all but finitely many such pairs, $C_G(A/B)$ is not contained in $H$, and derive a contradiction.  As $C_G(A/B)$ is a normal subgroup of $G$, by Theorem \ref{genob} there are only finitely many possibilities for the subgroup $C_G(A/B)$ that are not contained in $H$.  Thus the intersection $M$ of all such centralisers is open.  By the same argument as before, those $(A,B) \in \mcD$ for which $A \leq M$ account for every $\mcC$-composition factor of a finite image of $M$, thus if $P/Q$ is a chief factor of $M$ that is a direct power of a group in $\mcC$, then $[M,P] \leq Q$.

Let $N$ be the smallest normal subgroup of $M$ such that all the composition factors of $M/N$ are in $\mcC$.  Then $M/N$ is pronilpotent and $N$ is normal in $G$, so $G/N$ is virtually pronilpotent.  As $G$ is not virtually pronilpotent, it follows that $N$ has finite index in $G$.  There is $(A,B) \in \mcD$ such that $A$ is contained in $N$.  Note that $A/B$ is central in $N/B$ and hence abelian, say of exponent $p$, so $N/B$ has a chief factor of exponent $p$.  Moreover, all non-abelian finite simple groups appearing as composition factors of $N$ are outside of $\mcC$, and thus have Schur multipliers of order coprime to $p$.  Thus $E^p(N) < N$ by Lemma \ref{opsch}.  As $E^p(N)$ is characteristic in $N$, it is normal in $M$.  But then $M/E^p(N)$ is an image of $M$, all of whose composition factors are in $\mcC$, contradicting the definition of $N$.\end{proof}

\begin{lem}\label{centdec}Let $G$ be a centrally decomposable finite group, let $K$ be a non-central subgroup of $G$, and let $L$ be a normal subgroup of $G$.  Then there is a normal subgroup $H$ of $G$ and a maximal normal subgroup $M$ of $L$ such that $M$ does not contain $H$ and $H$ does not contain $K$.\end{lem}
\begin{proof}Let $G$ be the central product of proper subgroups $\{H_i \mid i \in I\}$.  Note that the $H_i$ are all normal in $G$, so by removing redundant generators, we may assume that for each $H_i$, there is a maximal normal subgroup of $G$ that does not contain $H_i$; if $L$ contains $H_i$, it follows that $M_iH_i = M_iL = G$, so $M_i \cap L$ is a maximal normal subgroup of $L$ that does not contain $H_i$.  As the $H_i$ are proper, $|I| \geq 2$.  Suppose $H_1$ and $H_2$ both contain $K$.  Then $H_1$ centralises $K$, since it centralises $H_2 \geq K$, and all other $H_i$ centralise $K$ since they centralise $H_1$.  As the $H_i$ generate $G$ it follows that $K \leq Z(G)$, a contradiction.\end{proof}

\begin{lem}[(e.g. \cite{Rei} Lemma 2.2; see also \cite{Zal})]\label{melfin}Let $G$ be a just infinite profinite group and let $H$ be an open subgroup of $G$.  Then $H$ has finitely many maximal open normal subgroups.\end{lem}

\begin{proof}[Proof of Theorem \ref{thmb}]We will obtain critical pairs $(R_n,S_n)$ in $G$ such that the subgroups $R_n$ have trivial intersection, and then use these to construct the required inverse system.

Set $R_{-1}=S_{-1}=GG$.  Suppose $R_n$ and $S_n$ have been chosen.  If $G$ is hereditarily just infinite, let $I$ be the intersection of all subgroups of $G$ normalised by $R_n$ that are not contained in $M$ as $M$ ranges over the maximal open normal subgroups of $R_n$.  Then $I$ is an open subgroup of $G$, as can be seen by applying Theorem \ref{genob} and then Lemma \ref{melfin} to each of the finitely many open subgroups of $G$ that contain $R_n$.  Otherwise, let $I = S_n$.  Now choose a critical pair $(A,B)$ as in Lemma \ref{seclem} such that $AC_G(A/B)$ is contained in $I$ and such that $A/B$ is a direct power of a group in $\mcC$.  Set $R_{n+1} = A$ and $S_{n+1} = B$.

Set $G_n = G/S_{n+1}$, set $A_n = R_n/S_{n+1}$, set $P_n = R_{n+1}/S_{n+1}$, and set $B_n = S_n/S_{n+1}$.  As $(R_n,S_n)$ in critical in $G$, it follows that $(A_n,B_n)$ is critical in $G_n$.  Moreover, $AC_G(A/B) \leq I \leq S_n$, so $P_n C_{G_n}(P_n)$ is contained in $B_n$ by construction.

Suppose $G$ is hereditarily just infinite, and let $T/S_{n+1}$ be a subgroup normalised by $A_n$ such that not all maximal normal subgroups of $A_n$ contain $T/S_{n+1}$.  Then $T \geq I$ by construction, so $T/S_{n+1}$ contains $P_nC_{G_n}(P_n)$.  As $P_n$ is not central in $A_n$, Lemma \ref{centdec} implies that no subgroup of $G_n$ containing $A_n$ is centrally decomposable.\end{proof}

\end{document}